\documentclass[11pt]{amsart}

\usepackage{amsmath}
\usepackage{amssymb}
\usepackage{amsfonts}
\usepackage{enumerate}
\usepackage{bbm}
\usepackage[mathscr]{eucal}
\usepackage{graphicx}
\usepackage{verbatim}

\usepackage{hyperref}
\hypersetup{colorlinks=true, urlcolor=black, linkcolor=black, citecolor=black}

\usepackage{dsfont}

\theoremstyle{plain}        \newtheorem{thm}{Theorem}
\theoremstyle{plain}        \newtheorem{pro}{Proposition}
\theoremstyle{plain}        \newtheorem{lem}{Lemma}
\theoremstyle{plain}        
\theoremstyle{plain}        
\theoremstyle{definition}   \newtheorem{defi}{Definition}

\begin{document}
	
\title[Fourier restriction to oscillating curves]{Restriction of the Fourier transform to some oscillating curves}
\author{Xianghong Chen \and Dashan Fan \and Lifeng Wang}
\address{X. Chen\\Department of Mathematical Sciences\\University of Wisconsin-Milwaukee\\Milwaukee, WI 53211, USA}
\email{chen242@uwm.edu}
\address{D. Fan\\Department of Mathematical Sciences\\University of Wisconsin-Milwaukee\\Milwaukee, WI 53211, USA}
\email{fan@uwm.edu}
\address{L. Wang\\Department of Mathematical Sciences\\University of Wisconsin-Milwaukee\\Milwaukee, WI 53211, USA}
\email{lifeng@uwm.edu, lifeng.wang.1987@gmail.com}


\begin{abstract} Let $\phi$ be a smooth function on a compact interval $I$. Let
$$\gamma(t)=\left (t,t^2,\cdots,t^{n-1},\phi(t)\right ).$$
In this paper, we show that
$$\left(\int_I \big|\hat f(\gamma(t))\big|^q \big|\phi^{(n)}(t)\big|^{\frac{2}{n(n+1)}} dt\right)^{1/q}\le C\|f\|_{L^p(\mathbb R^n)}$$
holds in the range
$$1\le p<\frac{n^2+n+2}{n^2+n},\quad 1\le q<\frac{2}{n^2+n}p'.$$
This generalizes an affine restriction theorem of Sj\"olin \cite{Sjoelin1974} for $n=2$. Our proof relies on ideas of Sj\"olin \cite{Sjoelin1974} and Drury \cite{Drury1985}, and more recently Bak-Oberlin-Seeger \cite{BakOberlinSeeger2008} and Stovall \cite{Stovall2016}, as well as a variation bound for smooth functions.
\end{abstract}

\maketitle

\section{Introduction}
Let $\gamma: I\rightarrow\mathbb R^n$ be a smooth curve and let $f$ be a Schwartz function on $\mathbb R^n$. We are interested in understanding restriction bounds of the form
\begin{equation}\label{eq:restriction}
\left(\int_I \big|\hat f(\gamma(t))\big|^q w(t) dt\right)^{1/q}\le C\|f\|_{L^p(\mathbb R^n)}
\end{equation}
where $\hat f$ is the Fourier transform of $f$, $C$ is a constant independent of $f$, and $p,q\ge 1$. In this context, it is natural to take
\begin{equation}\label{eq:affine-arclength}
w(t)=|\tau(t)|^{\frac{2}{n(n+1)}}
\end{equation}
where
$$\tau(t)=\det[\gamma'(t),\cdots,\gamma^{(n)}(t)]$$
is the torsion of $\gamma$. The image of the measure $|\tau(t)|^{\frac{2}{n(n+1)}}dt$ under $\gamma$ is called the \textit{affine arclength measure} of $\gamma$, which possesses several invariance properties. We refer the reader to Guggenheimer \cite{Guggenheimer1977} for more background on this notion. In what follows, unless otherwise stated, we will always assume that $w(t)=|\tau(t)|^{\frac{2}{n(n+1)}}$.

When $\gamma$ is nondegenerate, that is when $\tau(t)$ is nonvanishing on $I$, it has been shown by Sj\"olin \cite{Sjoelin1974} (for $n=2$) and Drury \cite{Drury1985} (for $n\ge 3$)
that the restriction bound \eqref{eq:restriction} holds for
\begin{equation}\label{eq:closed-range}
1\le p<\frac{n^2+n+2}{n^2+n},\quad 1\le q\le \frac{2}{n^2+n}p'
\end{equation}
provided that $I$ is a compact interval. Here $p'=\frac{p}{p-1}$ is the conjugate exponent of $p$. The range of $p$ is sharp, for example when $\gamma$ is the moment curve
$$\gamma(t)=(t,t^2,\cdots,t^n).$$
See Arkhipov, Chubarikov and Karatsuba \cite{ArkhipovChubarikovKaratsuba2004}. Sharpness of the range of $q$ follows from Knapp's homogeneity argument.

\begin{figure}[t]
\includegraphics[scale=0.5]{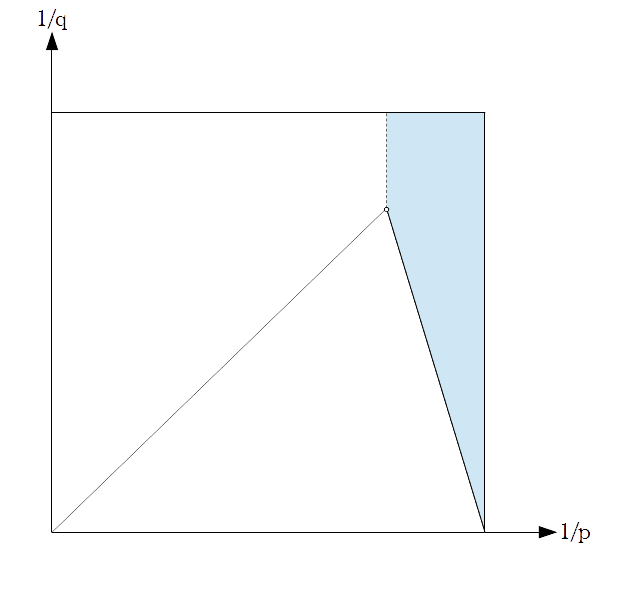}
\vspace{-0.7cm}
\caption{The shaded region corresponds to the range \eqref{eq:closed-range}.}
\end{figure}

For more general curves $\gamma$, it is harder to obtain \eqref{eq:restriction} in the whole range \eqref{eq:closed-range}. Partial results were obtained by Ruiz \cite{Ruiz1983}, Christ \cite{Christ1985}, Drury and Marshall \cite{DruryMarshall1985}, \cite{DruryMarshall1987} and Drury \cite{Drury1990} for some classes of finite-type curves. For monomial curves of the form
$$\gamma(t)=(t^{a_1},\cdots,t^{a_n})$$
(where $a_1,\cdots,a_n$ are distinct nonzero real numbers), sharp global results are due to Drury \cite{Drury1985} for $(a_1,\cdots,a_n)=(1,\cdots,n)$ and Bak, Oberlin and Seeger \cite{BakOberlinSeeger2009}, \cite{BakOberlinSeeger2013} for general exponents. In \cite{BakOberlinSeeger2013} a restricted strong type bound at the endpoint $p=\frac{n^2+n+2}{n^2+n}$ is also obtained. For curves that are perturbations of the monomial curves, sharp local results are due to Dendrinos and M\"uller \cite{DendrinosMueller2013}. Improving upon earlier work of Dendrinos and Wright \cite{DendrinosWright2008}, Stovall \cite{Stovall2016} obtained sharp global results for all polynomial curves, with bounds uniform over polynomials of given degree.

In this paper, we focus on \textit{smooth curves} of the form
\begin{equation}\label{eq:simple-curve}
\gamma(t)=\big(t,t^2,\cdots,t^{n-1},\phi(t)\big).
\end{equation}
For such curves one has $\tau(t)=C\phi^{(n)}(t)$ where $C>0$ is a dimensional constant.
Following \cite{DruryMarshall1985} and \cite{BakOberlinSeeger2008}, we will call such curves \textit{simple curves}. Note that if $\phi$ is so that $\gamma$ is of finite type
or is a perturbation of a monomial curve, sharp restriction bounds for $\gamma$ follow from the work of Dendrinos and M\"uller \cite{DendrinosMueller2013} mentioned above. If $\phi$ is a polynomial, a uniform bound at the endpoint $p=\frac{n^2+n+2}{n^2+n}$ is obtained by Bak, Oberlin and Seeger \cite{BakOberlinSeeger2013}.

In the case $n=2$, Sj\"olin \cite{Sjoelin1974} obtained sharp uniform bounds for all simple curves with convex $\phi$. For $n\ge 3$, Bak, Oberlin and Seeger \cite{BakOberlinSeeger2008} identified some general conditions on
$\phi$ that imply sharp uniform bounds. Their conditions in particular require $\phi^{(n)}(t)$ to be nondecreasing and to satisfy a geometric inequality.
As a consequence, one obtains sharp restriction bounds for the simple curves with
$$\phi(t)=e^{-t^{-\alpha}},\ \alpha>0.$$
See \cite[Section~4]{BakOberlinSeeger2008} for details.

For general simple curves in $n=2$ defined on compact $I$, Sj\"olin \cite{Sjoelin1974} showed that \eqref{eq:restriction} holds in the whole range \eqref{eq:closed-range} if one instead takes $w(t)=w_{\varepsilon}(t)$, where
$$w_{\varepsilon}(t):=|\tau(t)|^{\frac{2}{n(n+1)}+\varepsilon},\ \varepsilon>0.$$
Moreover, he showed that \eqref{eq:restriction} fails
with $w(t)=w_0(t)$ if $q=\frac{2}{n^2+n}p'<\infty$ and
\begin{equation}\label{eq:alpha-beta-curve}
\phi(t)=e^{-t^{-\alpha}}\sin(t^{-\beta})
\end{equation}
where $\alpha,\beta>0$ and $\beta$ is sufficiently large.

Our main result extends Sj\"olin's result to $n\ge 3$.\footnote{The same statement holds true for $n=2$ and $n=1$.}

\begin{thm}\label{intro:thm}
Let $n\ge 3$ and let $\gamma$ be a simple curve as in \eqref{eq:simple-curve} with $\phi$ defined on a compact interval $I$. Then the restriction bound \eqref{eq:restriction} holds in the following cases:

(i) $w(t)=w_\varepsilon(t)$, $\varepsilon>0$ and
\begin{equation}\label{eq:range-1}
1\le p<\frac{n^2+n+2}{n^2+n},\quad 1\le q\le\frac{2}{n^2+n}p';
\end{equation}

(ii) $w(t)=w_0(t)$ and
\begin{equation}\label{eq:range-2}
1\le p<\frac{n^2+n+2}{n^2+n},\quad 1\le q<\frac{2}{n^2+n}p'.
\end{equation}
\end{thm}

Case $(ii)$ is not explicitly formulated in Sj\"olin \cite{Sjoelin1974}, but follows easily from the treatment for case $(i)$. See also Drury and Marshall \cite[Theorem~1]{DruryMarshall1985} for a similar formulation.

The sharpness of Theorem \ref{intro:thm} can be justified by Sj\"olin's oscillating curves.

\begin{pro}\label{intro:prop}
Let $n\ge 2$ and $\alpha,\beta>0$. Let $\gamma$ be the simple curve as in \eqref{eq:simple-curve} with $\phi$ given by \eqref{eq:alpha-beta-curve} and defined on $I=[0,1]$. Then the restriction bound \eqref{eq:restriction} holds in cases $(i)$ and $(ii)$ of Theorem \ref{intro:thm}. Moreover, \eqref{eq:restriction} fails if $w(t)=w_0(t)$, $q=\frac{2}{n^2+n}p'<\infty$, and $\beta>\frac{(n+1)}{2}\alpha$.
\end{pro}

Sj\"olin's result in $n=2$ relies on uniform restriction bounds for simple curves with convex $\phi$. In $n\ge 3$, we deduce Theorem \ref{intro:thm} based on uniform restriction bounds for simple curves with essentially constant torsion.

\begin{lem}\label{intro:lemma}
Let $n\ge 3$ and let $\gamma$ be a simple curve as in \eqref{eq:simple-curve} with $\phi$ satisfying
\begin{equation}\label{eq:nondegenerate}
1/2\le |\phi^{(n)}(t)|\le 1,\ t\in I.
\end{equation}
Then the restriction bound \eqref{eq:restriction} holds for
\begin{equation}\label{eq:endline}
1\le p<\frac{n^2+n+2}{n^2+n},\quad
q=\frac{2}{n^2+n}p'
\end{equation}
and for a constant $C$ depending only on $p$ and $n$.
\end{lem}

It is important to note that the constant $C$ is independent of $\phi$. Drury \cite[Theorem~2]{Drury1985} has obtained a similar result for all smooth curves defined on compact $I$, but with the constant $C$ depending on other information of the curve (besides the assumption \eqref{eq:nondegenerate}). Such dependence cannot be avoided in Drury's result as can be seen by taking $\gamma$ to be a nondegenerate closed curve (for the existence such curves, see Costa \cite{RodriguesCosta1990}). Note also that Lemma \ref{intro:lemma} follows from Theorem 1 in Bak, Oberlin and Seeger \cite{BakOberlinSeeger2009} if $\phi$ satisfies some additional mild assumptions.

With the uniform restriction bounds for convex simple curves in $n=2$, Sj\"olin deduced his result for general simple curves using the following lemma.

\begin{lem}\label{intro:sjolin-lemma}{\cite[Lemma~1]{Sjoelin1974}}
Let $\varphi$ be a smooth function on a compact interval $I$. Let $E=\{t\in I: \varphi(t)=0\}$ and let $\{I_k\}_{k=1}^\infty$ be the connected components of $I\backslash E$. Then
$$\sum_{k=1}^{\infty} \big(\sup_{I_k}|\varphi|\big)^\delta<\infty,\ \forall\delta>0.$$
\end{lem}

In $n\ge 3$, we deduce Theorem \ref{intro:thm} from Lemma \ref{intro:lemma} using the following lemma.

\begin{lem}\label{intro:variation}
Let $\varphi$ be a smooth function on a compact interval $I$. For $k\in\mathbb Z$, let $E_k=\{t\in I: 2^{-k-1}\le |\varphi(t)|\le 2^{-k}\}$. Then there exist intervals $\{I_{k,j}\subset I\}_{j=1}^{N_k}$ such that
$$E_k\subset \bigcup_{j=1}^{N_k} I_{k,j}$$
and such that
$$2^{-k-2}\le |\varphi(t)|\le 2^{-k+1},\ t\in I_{k,j};$$
moreover, $N_k$ satisfies
$$N_k\le C_\delta 2^{\delta k},\ \forall \delta>0.$$
\end{lem}

Note that Lemma \ref{intro:sjolin-lemma} follows from Lemma \ref{intro:variation}, but not vice versa. More general versions of Lemma \ref{intro:variation} and Theorem \ref{intro:thm} under weaker assumptions are stated in Section \ref{variation} and Section \ref{theorem} respectively.

The proof of Lemma \ref{intro:lemma} follows Drury's argument in \cite{Drury1985} using offspring curves. We are able to obtain uniform bounds here because the special form of the simple curves allows several technical steps in the iteration process to go through uniformly. The proof of Theorem \ref{intro:thm} proceeds by decomposing the curve into segments according to the size of $\phi^{(n)}(t)$ and then summing up the pieces directly using Lemma \ref{intro:lemma} and rescaling. Such an approach has also been applied to polynomial curves in Stovall \cite{Stovall2016}, where it is coupled with a square function estimate so that the pieces are summed up in a more efficient way.

The paper is organized as follows. In Section \ref{lemma} we give a detailed proof of Lemma \ref{intro:lemma}. In Section \ref{variation} we prove a more general version of Lemma \ref{intro:variation} under weaker smoothness assumptions. In Section \ref{theorem} we prove a more general version of Theorem \ref{intro:thm}  based on Sections \ref{lemma} and \ref{variation}. In Section \ref{prop} we prove Proposition \ref{intro:prop} by examining Knapp type examples. A technical calculation needed in Section \ref{prop} is postponed to Section \ref{appendix}. Throughout the paper, $C$ denotes a constant whose value may change from line to line.

\textit{Acknowledgment.}
We would like to thank Andreas Seeger for bringing this subject to our attention and for many useful suggestions.

\section{Proof of Lemma \ref{intro:lemma}}\label{lemma}

In this section, we give a detailed proof of Lemma \ref{intro:lemma} following Drury's argument in \cite{Drury1985}. We will assume that $I=[a,b]$ is a compact interval and $\phi\in C^{n}(I)$ satisfies \eqref{eq:nondegenerate}.

Associated with the given curve 
$$\gamma(t)=\big(t,t^2,\cdots,t^{n-1},\phi(t)\big),\ t\in I,$$
define a family of curves
\begin{equation*}
\Upsilon=\left \{\gamma_{\alpha}:\gamma_{\alpha}(t)=\frac{1}{N}\sum_{k=1}^{N}\gamma(t+\alpha_k)  \right \}
\end{equation*}
where $N\in\mathbb N$ and $\alpha=(\alpha_1,\cdots,\alpha_N)\in\mathbb R^N$ satisfies $0\leq\alpha_1\leq\cdots\leq\alpha_N$. The domain of $\gamma_{\alpha}$ is $I_{\alpha}=[a-\alpha_1,b-\alpha_N]$, so that each $\gamma(t+\alpha_k)$ is well defined. Note that $\gamma\in \Upsilon$. A curve $\gamma_{\alpha}\in\Upsilon$ takes the form
\begin{equation*}
\gamma_{\alpha}(t)=\left (t+\frac{1}{N}\sum_{k=1}^{N}\alpha_k,\cdots,\frac{1}{N}\sum_{k=1}^{N}(t+\alpha_k)^{n-1},
\Phi_{\alpha}(t)\right )
\end{equation*}
where
\begin{equation*}
\Phi_{\alpha}(t)=\frac{1}{N}\sum_{k=1}^{N}\phi(t+\alpha_k)
\end{equation*}
satisfies
\begin{equation}\label{eq:nondegenerate-Phi}
1/2\le |\Phi_\alpha^{(n)}(t)|\le 1,\ t\in I_\alpha.
\end{equation}
To prove Lemma \ref{intro:lemma}, we will show that
\begin{equation}\label{unweighted}
\left (\int_{I_\alpha}\left |\hat{f}(\gamma_\alpha(t))
\right |^{q}dt\right )^{{1}/{q}}
\leq C_{p,n}\|f\|_{L^{p}(\mathbb{R}^n)}
\end{equation}
holds uniformly for all $\gamma_\alpha\in \Upsilon$, where $p, q$ satisfy \eqref{eq:endline} and $C_{p,n}$ depends only on $p$ and $n$. For simplicity, we will denote $\gamma=\gamma_\alpha$, $I=I_\alpha$. 

By duality, it suffices to show that the operator 
\begin{equation*}
\mathcal E(g)(x):=\int_I e^{i \gamma(t)\cdot x}g(t)dt
\end{equation*}
satisfies
\begin{equation} \label{eq3}
\|\mathcal E(g)\|_{L^{p'}(\mathbb R^n)}\leq C_{p,n}\|g\|_{L^{q'}(I)}
\end{equation}
for
\begin{equation} \label{eq3-range}
\frac{n^2+n}{2}< p' \le\infty,\quad \frac{n^2+n}{2}\frac{1}{p'}+\frac{1}{q'}=1.
\end{equation}
The proof is by induction on $p'$. The induction hypothesis is that for some $p'_0$ and $q'_0$ with
$$\frac{n^2+n}{2}\frac{1}{p'_0}+\frac{1}{q'_0}=1,$$
we have
\begin{equation} \label{eq4}
\|\mathcal E(g)\|_{L^{p'_0}(\mathbb R^n)}\leq C_{p_0,n}\|g\|_{L^{q'_0}(I)},
\end{equation}
that is,
\begin{equation}\label{eq31}
\left (\int_{\mathbb{R}^n}\left| \int_{I_{\alpha}}e^{i\gamma_{\alpha}(t)\cdot x}g(t)dt\right |^{p'_0} dx\right )^{{1}/{p'_0}}\leq
C_{p_0,n}\left (\int_{I_{\alpha}}|g(t)|^{q'_0}dt
\right )^{{1}/{q'_0}}
\end{equation}
holds uniformly for $\gamma_{\alpha}\in\Upsilon$. The base case is that (\ref{eq4}) holds for $p'_0=\infty$ and $q'_0=1$, with $C=1$.

By Fubini's theorem, we can write
\begin{align*}
\big ( \mathcal E(g) (x)
\big )^n
&=\left (\int_I e^{i\gamma(t)\cdot x}g(t) dt\right )^n \\
&=\int_{I^n}e^{i\frac{\sum_{k=1}^{n}\gamma(t_k)}{n}\cdot nx}
\prod_{k=1}^{n}g(t_k)dt_1\cdots dt_n.
\end{align*}
Since the last integral is symmetric in $t_1,\cdots,t_n$, we have
\begin{equation}\label{eq:symmetry}
\big ( \mathcal E(g) (x)\big )^n
=n!\int_{A}e^{i\frac{\sum_{k=1}^{n}\gamma(t_k)}{n}\cdot nx}
\prod_{k=1}^{n}g(t_k)dt_1\cdots dt_n
\end{equation}
where
\begin{equation*}
A =\{(t_1,\cdots,t_n)\in I^n: t_1 < t_2<\cdots < t_n\}.
\end{equation*}

Apply the change of variables
\begin{align}\label{eq:t-to-th}
&t=t_1,\quad h_k=t_k-t_1,\ k=2,\cdots, n.
\end{align}
From \eqref{eq:symmetry} we can write
\begin{equation}\label{eq:symmetry-h}
\big ( \mathcal E(g) (x)\big )^n
=n!\int_{B}e^{i\frac{\sum_{k=1}^{n}\gamma(t+h_k)}{n}\cdot nx}
\prod_{k=1}^{n}g(t+h_k) dtdh_2\cdots dh_n
\end{equation}
where $h_1\equiv 0$ and $B$ is the image of $A$ under the change of variables \eqref{eq:t-to-th}. Note that the curves
$$t\mapsto\frac{1}{n} \sum_{k=1}^{n}\gamma(t+h_k),$$
belong to the family $\Upsilon$.

Write
$$v(h)=h_2\cdots h_n \prod_{2\le i<j\le n}(h_j-h_i)$$
and define
\begin{equation}\label{eq:T}
T(F)(x)
=\int_{B} 
e^{i\frac{\sum_{k=1}^{n}\gamma(t+h_k)}{n}\cdot x} F(t,h)v(h) dtdh.
\end{equation}
By Minkowski's inequality and the induction hypothesis \eqref{eq31},
\begin{align}\label{eq:minkowski}
\|T(F)\|_{L^{p'_0}(dx)}
&\le \int \left \| \int
e^{i\frac{\sum_{k=1}^{n}\gamma(t+h_k)}{n}\cdot x} F(t,h) dt\right \|_{L^{p'_0}(dx)}v(h)dh\\
&\le C_{p_0,n} \int \left \| F(\cdot,h) \right \|_{L^{q'_0}(I_h,dt)}v(h)dh.\notag\\
&\le C_{p_0,n} \left \| F \right \|_{L^1(vdh,L^{q'_0}(dt))}.\notag
\end{align}

On the other hand, consider the change of variables	\begin{equation}\label{eq:change-of-variables}
y=\frac{\sum_{k=1}^{n}\gamma(t+h_k)}{n}.
\end{equation}
The corresponding Jacobian is
\begin{equation*}
J(t,h)={\frac{1}{n^n}}\left |\det[\gamma'(t),\gamma'(t+h_2) \cdots, \gamma'(t+h_n)]\right |.
\end{equation*}
By a generalized Rolle's theorem (see Exercises 95 and 96 in Part V, Chapter 1 of \cite{PolyaSzego1998}), we have
\begin{equation*}
\det[\gamma'(t),\gamma'(t+h_2) \cdots, \gamma'(t+h_n)]
=\frac{\Phi_\gamma^{(n)}(\xi)}{n!} h_2\cdots h_n \prod_{2\le i<j\le n}(h_j-h_i)
\end{equation*}
for some $\xi\in(t+h_{n-1},t+h_n)$. Therefore, by \eqref{eq:nondegenerate-Phi}, 
\begin{equation}\label{eq:jacobian}
J(t,h)\ge C_n v(h).
\end{equation}
In particular, $J$ is nonvanishing on $B$. It follows that (see \cite[Prop.~3.9]{DendrinosStovall2015}) the change of variables \eqref{eq:change-of-variables} is injective. So we can write \eqref{eq:T} as
\begin{align*}
T(F)(x)
&=n!\int_{\widetilde{A }} e^{iy\cdot x} F(t,h) \frac{v(h)}{J(t,h)}dy
\end{align*}
where $\widetilde{A }$ is the image of $B$ under the change of variables \eqref{eq:change-of-variables}. By the Plancherel theorem, we have
$$\|T(F)\|_{L^{2}(dx)}
=C_n \left (\int_{\widetilde{A }} \left |F(t,h)\right |^2
\left [ \frac{v(h)}{J(t,h)}\right ]^2dy\right )^{1/2}.$$
Changing the variables back and using \eqref{eq:jacobian}, we get
\begin{align}\label{eq:plancherel}
\|T(F)\|_{L^{2}(dx)}
&= C_n \left (\int_{B} \left |F(t,h)\right |^2
\left [\frac{v(h)}{J(t,h)}\right ]v(h)dtdh\right )^{1/2}\\
&\le C_n \left (\int_{B} \left |F(t,h)\right |^2 v(h)dtdh\right )^{1/2}\notag\\
&\le C_n \left \| F \right \|_{L^2(vdh,L^{2}(dt))}.\notag
\end{align}

By an interpolation argument (see \cite{BenedekPanzone1961}), from \eqref{eq:minkowski} and \eqref{eq:plancherel} we obtain
\begin{equation}\label{eq:interpolate}
\|T(F)\|_{L^{u}(dx)}
\le C_{p_0,n} \left \| F \right \|_{L^r(vdh,L^{s}(dt))}
\end{equation}
for any
\begin{equation}  \label{eq19}
\left (\frac{1}{r},\frac{1}{s},\frac{1}{u}\right )=(1-\theta)\left (1,\frac{1}{q'_0},\frac{1}{p'_0}\right )+\theta\left (\frac{1}{2},\frac{1}{2},\frac{1}{2}\right )
\end{equation}
with $\theta\in (0,1)$, and for a constant $C_{p_0,n}$ depending only on $p_0$ and $n$. In particular, taking $$F(t,h)=\frac{1}{v(h)}\prod_{k=1}^{n}g(t+h_k)$$ 
in \eqref{eq:interpolate}, we obtain, by \eqref{eq:symmetry-h},
\begin{align}\label{eq:Lnu}
\|\mathcal E(g)\|_{L^{nu}(dx)}^n
\le C_{p_0,n}  \left (\int G(h)^{r/s} v^{1-r}(h)dh\right )^{1/r}.
\end{align}
where
$$G(h)=\int \Big | \prod_{k=1}^{n}g(t+h_k)\Big |^s dt.$$

By \cite[Lemma~1]{DruryMarshall1985}, we have
$$v^{-\frac{2}{n}}\in L^{1,\infty}(dh).$$
Therefore, with
\begin{equation}\label{r<1+2/n}
1<r<1+\frac{2}{n}
\end{equation}
and
$$\rho=\frac{2}{n(r-1)}>1,$$
we can bound
\begin{align}\label{eq:Gv}
\int G(h)^{r/s} v^{1-r}(h)dh
&\le C_\rho\|G^{r/s}\|_{L^{\rho',1}(dh)}\|v^{1-r}\|_{L^{\rho,\infty}(dh)}\\
&=C_\rho\|G^{r/s}\|_{L^{\rho',1}(dh)}\|v^{-\frac{2}{n}}\|^{1/\rho}_{L^{1,\infty}(dh)}\notag \\
&\le C_{r,n}\|G^{r/s}\|_{L^{\rho',1}(dh)}.\notag
\end{align}

If for some $E\subset I$,
\begin{equation}\label{eq:restricted-g}
|g(t)|\le 1_E(t),\ t\in I.
\end{equation}
Then
$$G(h)\le |E|,\quad \int G(h) dh\le |E|^n.$$
It follows that
\begin{equation}\label{eq:holder}
\|G^{r/s}\|_{L^{\rho',1}(dh)}
\le C_{r,n} |E|^{\frac{r}{s}+\frac{n-1}{\rho'}}
\end{equation}
provided
\begin{equation}\label{eq:rsrho}
\frac{r}{s}\rho'>1.
\end{equation}

Combining \eqref{eq:Lnu}, \eqref{eq:Gv} and \eqref{eq:holder}, we obtain
$$\|\mathcal E(g)\|_{L^{nu}(dx)}
\le  C_{p_0,r,n} |E|^{\frac{1}{ns}+\frac{n-1}{nr\rho'}}$$
as long as \eqref{eq19}, \eqref{r<1+2/n}, \eqref{eq:rsrho} and \eqref{eq:restricted-g} are satisfied. By interpolation, this produces a new bound \eqref{eq3} for any $p',q'$ satisfying \eqref{eq3-range} with
$$p'>p'_1,\quad q'<q'_1,$$
where
$$\left (\frac{1}{p'_1},\frac{1}{q'_1}\right )
=\frac{n-2}{n(n+2)}\left (\frac{1}{p'_0},\frac{1}{q'_0}\right )
+\left (\frac{2}{n(n+2)},\frac{2}{n(n+2)}\right ).$$

Iterating this process, we see that \eqref{eq3} holds for $p',q'$ in the full range \eqref{eq3-range}. This completes the proof of Lemma \ref{intro:lemma}.

\section{Proof of Lemma \ref{intro:variation}}\label{variation}
We now prove Lemma \ref{intro:variation}. Let $\varphi$ be a continuous (real-valued) function defined on a compact interval $[a,b]$. Let $r>0$. Suppose the set
$$E_r=\{t\in [a,b]: r/2 \le |\varphi(t)|\le r\}\neq\emptyset.$$
For $t\in E_r$, define
$$a_t=\sup\,\{s\in [a,t]: |\varphi(s)|\le r/4 \text{ or } |\varphi(s)|\ge 2r\}$$
if the set on the right-hand side is not empty; otherwise define $a_t=a$. Similarly, define
$$b_t=\inf\,\{s\in [t,b]: |\varphi(s)|\le r/4 \text{ or } |\varphi(s)|\ge 2r\}$$
if the set on the right-hand side is not empty; otherwise define $b_t=b$. Let
\begin{align*}
I_t=(a_t,b_t)\quad \text{if } a_t>a \text{ and } b_t<b,\\
I_t=[a_t,b_t)\quad \text{if } a_t=a \text{ and } b_t<b,\\
I_t=(a_t,b_t]\quad \text{if } a_t>a \text{ and } b_t=b,\\
I_t=[a_t,b_t]\quad \text{if } a_t=a \text{ and } b_t=b.
\end{align*}
Note that $\varphi$ has a definite sign on $I_t$.

\begin{lem}\label{lem:disjoint}
For any $t_1, t_2\in E_r$, we have either $I_{t_1}=I_{t_2}$ or $I_{t_1}\cap I_{t_2}=\emptyset$.
\end{lem}

\begin{proof}
If $t_2\in I_{t_1}$, then by the definitions of $I_{t_1}$ and $I_{t_2}$, we must have $I_{t_1}=I_{t_2}$. If $t_2\notin I_{t_1}$, then $I_{t_1}\cap I_{t_2}=\emptyset$.
\end{proof}

Notice that $\{I_t\}_{t\in E_r}$ forms an open cover of the compact set $E_r$ in $[a,b]$. So there must be a finite subcover. Combining with Lemma \ref{lem:disjoint}, we get the following.

\begin{lem}\label{lem:finite}
$\{I_t\}_{t\in E_r}$ is a finite set of disjoint intervals.
\end{lem}

With Lemma \ref{lem:finite}, we can make the following.

\begin{defi}
Let $\varphi$ be a continuous function defined on a compact interval $[a,b]$. For $r>0$, define
$$N(r;\varphi)=\#\{I_t\}_{t\in E_r}$$
where $\{I_t\}_{t\in E_r}$ is as described above.
\end{defi}

We now prove the main technical lemma. Denote by $C^1[a,b]$ the space of functions on $[a,b]$ whose derivative is continuous on $[a,b]$.

\begin{lem}\label{lem:first-variation}
Suppose $\varphi\in C^1[0,1]$ and $N(r;\varphi)\ge \mathcal N\ge 20$. Then
$$N\left (\frac{\mathcal N}{8}r;\varphi'\right )\ge \frac{\mathcal N}{16}.$$
\end{lem}
\begin{proof}
Since the intervals $I_t$ corresponding to $\varphi$  and $r$ are disjoint, there must be at least $\mathcal N/2$ many of them that satisfy
$$|I_t|\le \frac{2}{\mathcal N}.$$
After perhaps excluding one such interval (the rightmost one), each of them has its right endpoint satisfying either $\varphi=\pm 2r$ or $\varphi=\pm r/4$.

Without loss of generality, assume that at least a quarter of these right endpoints satisfy $\varphi=r/4$. Denote them by $b_1<\cdots<b_M$, with
$$M>\frac{\frac{\mathcal N}{2}-1}{4}.$$
In each of the corresponding intervals there must be a point at which $\varphi=r/2$. Denote them by $c_1<\cdots<c_M$. By the mean value theorem, for $j=1,\cdots,M$, there exists $\xi_j\in (c_j,b_j)$ at which
$$\varphi'(\xi_j)\le\frac{\frac{r}{4}-\frac{r}{2}}{\frac{2}{\mathcal N}}=-\frac{\mathcal N}{8}r.$$
On the other hand, for $j=1,\cdots,M-1$, there exists $\eta_j\in (b_j,c_{j+1})$ at which
$$\varphi'(\eta_j)\ge\frac{\frac{r}{2}-\frac{r}{4}}{c_{j+1}-b_j}\ge 0.$$
Therefore, we must have
$$N\left (\frac{\mathcal N}{8}r;\varphi'\right )\ge M-1.$$
Since $\mathcal N\ge 20$, we have
$$M-1>\frac{\frac{\mathcal N}{2}-1}{4}-1\ge \frac{\mathcal N}{16},$$
and the conclusion follows.
\end{proof}

We also need the following lemma.

\begin{lem}\label{lem:first-variation-alpha}
Suppose $\varphi\in C^\alpha[0,1]$ with $\alpha\in(0,1]$. Then $$N(r;\varphi)\le C r^{-1/\alpha},\ r>0.$$
\end{lem}

\begin{proof}
Denote $\mathcal N=N(r;\varphi)$. As in the proof of Lemma \ref{lem:first-variation}, there must be at least $\mathcal N/2$ many of the $I_t$'s that satisfy
$$|I_t|\le \frac{2}{\mathcal N}.$$
Without loss of generality we may assume that $\mathcal N\ge 2$. It follows that there is at least one such $I_t$ on which
$$\sup_{I_t}\varphi - \inf_{I_t}\varphi \ge r/4.$$
On the other hand, since $\varphi\in C^\alpha[0,1]$,
$$\sup_{I_t}\varphi - \inf_{I_t}\varphi\le C |I_t|^\alpha\le C \left(\frac{2}{\mathcal N}\right)^\alpha.$$
Combining, we obtain
$$\mathcal N\le C r^{-1/\alpha},$$
as desired.
\end{proof}

Iterating Lemma \ref{lem:first-variation}, we see that, if $\varphi\in C^\ell[0,1]$ ($\ell\in\mathbb N$), then
\begin{equation}\label{eq:iterate}
N\left (\frac{\mathcal N^\ell}{8^\ell 4^{\ell(\ell-1)}}r;\varphi^{(\ell)}\right )\ge \frac{\mathcal N}{16^\ell}
\end{equation}
provided
\begin{equation}\label{eq:iterate-condition}
N(r;\varphi)\ge \mathcal N\ge 20\cdot 16^{\ell-1}.
\end{equation}
Combining this with Lemma \ref{lem:first-variation-alpha}, we obtain the following (which contains Lemma \ref{intro:variation} as a corollary).

\begin{lem}\label{lem:tempered-variation}
Suppose $\varphi\in C^{\alpha}[a,b]$ with $\alpha\in (0,\infty)$. Then $$N(r;\varphi)\le C r^{-1/\alpha},\ r>0.$$
In particular, if $\varphi\in C^\infty[a,b]$, then
$$N(2^{-k};\varphi)\le C_\delta 2^{\delta k},\ \forall \delta>0.$$
\end{lem}

\begin{proof}
By an affine change of variable we may assume that $[a,b]=[0,1]$. The case $\alpha\le 1$ follows from Lemma \ref{lem:first-variation-alpha}. If $\alpha\in (1,\infty)$, we can write
$$\alpha=\ell+\beta$$
where $\ell\in\mathbb N$ and $\beta\in (0,1]$. Let $\mathcal N=N(r;\varphi)$. Since $\varphi$ is bounded, we may assume without loss of generality that \eqref{eq:iterate-condition} is satisfied. By \eqref{eq:iterate} and Lemma \ref{lem:first-variation-alpha}, it follows that
$$\frac{\mathcal N}{16^\ell}\le C\left (\frac{\mathcal N^\ell}{8^\ell 4^{\ell(\ell-1)}}r\right)^{-1/\beta}.$$
From this we obtain (with a different constant $C$)
$$\mathcal N\le C r^{-1/\alpha},$$
as desired.
\end{proof}

\section{Proof of Theorem \ref{intro:thm}}\label{theorem}
We now prove a more general version Theorem \ref{intro:thm} using Lemma \ref{intro:lemma} and Lemma \ref{lem:tempered-variation}. We first make the observation that Lemma \ref{intro:lemma} remains valid if the condition \eqref{eq:nondegenerate} is replaced by
$$1/4\le |\phi^{(n)}(t)|\le 2,\ t\in I,$$
since the proof of Lemma \ref{intro:lemma} does not rely on the exact values of the constant bounds. Observe also that, for $r>0$, by considering the following affine image of the original curve
$$\gamma_r(t)=\big(t,t^2,\cdots,t^{n-1},r^{-1}\phi(t)\big),$$
Lemma \ref{intro:lemma} implies
\begin{equation}\label{eq:r-torsion}
\left(\int_I \big|\hat f(\gamma(t))\big|^q dt\right)^{1/q}\le C\, r^{-\frac{1}{p'}}\|f\|_{L^p(\mathbb R^n)}
\end{equation}
in the same rage \eqref{eq:endline}, provided
$$r/4\le |\phi^{(n)}(t)|\le 2r,\ t\in I.$$
sFurthermore, by H\"older's inequality, the range \eqref{eq:endline} can be extended to \eqref{eq:range-1} if $I$ a finite interval.

Now suppose $I$ is a compact interval, $\varepsilon>-\frac{2}{n^2+n}$, and $p, q$ satisfy \eqref{eq:range-1}. Then we can write
$$\int_I \big|\hat f(\gamma(t))\big|^q w_\varepsilon(t) dt
=\sum_{k\ge k_0} \int_{E_k} \big|\hat f(\gamma(t))\big|^q w_\varepsilon(t) dt$$
where $k_0\in\mathbb Z$ depends only on $\|\phi^{(n)}\|_\infty$ and
$$E_k=\{t\in I: 2^{-k-1}<|\phi^{(n)}(t)|\le 2^{-k}\}.$$
For every fixed $k$, let $\{I_{k,j}\}_{j=1}^{N_k}$ be as in Section \ref{variation} associated with $\varphi=\phi^{(n)}$ and $r=2^{-k}$. Then we can bound
\begin{align*}
\int_{E_k} \big|\hat f(\gamma(t))\big|^q w_\varepsilon(t) dt
&\le \sum_{j=1}^{N_k} \int_{I_{k,j}} \big|\hat f(\gamma(t))\big|^q w_\varepsilon(t) dt\\
&\le C 2^{-k\big(\frac{2}{n^2+n}+\varepsilon\big)}
\sum_{j=1}^{N_k} \int_{I_{k,j}} \big|\hat f(\gamma(t))\big|^q dt.
\end{align*}
By \eqref{eq:r-torsion}, we can bound each
$$\int_{I_{k,j}} \big|\hat f(\gamma(t))\big|^q dt\le C 2^{k\frac{q}{p'}} \|f\|_{L^p(\mathbb R^n)}^q.$$
Therefore,
$$\int_I \big|\hat f(\gamma(t))\big|^q w_\varepsilon(t) dt
\le C \left(\sum_{k\ge k_0} N_k 2^{-k\big(\frac{2}{n^2+n}+\varepsilon\big)} 2^{k\frac{q}{p'}}\right) \|f\|_{L^p(\mathbb R^n)}^q.$$
Combining this with Lemma \ref{lem:tempered-variation}, we obtain the following.

\begin{thm}\label{lem:Holder}
Let $I$ be a compact interval and let $\alpha>n$.
Suppose $\gamma$ is a simple curve as in \eqref{eq:simple-curve} with $\phi\in C^\alpha(I)$. Then the restriction bound \eqref{eq:restriction} holds with $w(t)=w_\varepsilon(t)$, provided
\begin{equation}\label{eq:range-eps-A}
\frac{2}{n^2+n}+\varepsilon>\frac{q}{p'}+\frac{1}{\alpha-n}
\end{equation}
and
\begin{equation}\label{eq:range-A}
1\le p<\frac{n^2+n+2}{n^2+n},\quad 1\le q\le\frac{2}{n^2+n}p'.
\end{equation}
\end{thm}

Applying Theorem \ref{lem:Holder} with $\alpha\to\infty$, this proves Theorem \ref{intro:thm}.

\section{Proof of Proposition \ref{intro:prop}}\label{prop}
We now prove Proposition \ref{intro:prop}. It is clear that for any $\alpha,\beta>0$,
\begin{equation}\label{eq:alpha-beta-phi}
\phi(t)=e^{-t^{-\alpha}}\sin(t^{-\beta})\in C^\infty(I)
\end{equation}
where $I=[0,1]$. Therefore, by Theorem \ref{intro:thm} the restriction bound \eqref{eq:restriction} holds in cases $(i)$ and $(ii)$ of Theorem \ref{intro:thm}. 

Now assume $w(t)=w_0(t)$, $p>1$, and $q=\frac{2}{n^2+n}p'$. We are going to show that \eqref{eq:restriction} fails for $\beta>\frac{(n+1)}{2}\alpha$. We do this by examining Knapp type examples. Fix a function $\chi\in C^\infty_c(\mathbb R)$ that satisfies $\chi(t)=1,\ |t|\le 1$. For $\delta\in(0,1)$, let
$$\widehat f_\delta(\xi)=\chi\left (\frac{\xi_1}{\delta}\right )
\chi\left (\frac{\xi_2}{\delta^2}\right )
\cdots
\chi\left (\frac{\xi_{n-1}}{\delta^{n-1}}\right )
\chi\left (\frac{\xi_{n}}{e^{-\delta^{-\alpha}}}\right ).$$
By a simple dilation argument, we have
$$\|f_\delta\|_{L^p(\mathbb R^n)}
\le C \delta^{1/p'}\delta^{2/p'}\cdots\delta^{(n-1)/p'} e^{-\delta^{-\alpha}/p'}$$
On the other hand, since $|\sin(t^{-\beta})|\le 1$, we have
$$\widehat f_\delta(\gamma(t))=1,\ 0\le t\le \delta$$
and therefore
$$\left(\int_0^{1} \big|\widehat f_\delta(\gamma(t))\big|^q \big|\phi^{(n)}(t)\big|^{\frac{2}{n(n+1)}} dt\right)^{1/q}
\ge \left(\int_0^{\delta} \big|\phi^{(n)}(t)\big|^{\frac{2}{n(n+1)}} dt\right)^{1/q}.$$
Combining these, we see that, if \eqref{eq:restriction} holds, then
\begin{equation}\label{eq:knapp}
\int_0^{\delta} \big|\phi^{(n)}(t)\big|^{\frac{2}{n(n+1)}} dt
\le C \delta^{\frac{n-1}{n+1}} e^{-\frac{2}{n(n+1)}\delta^{-\alpha}}.
\end{equation}

To estimate the left-hand side, we will prove the following lemma in the Appendix.

\begin{lem}\label{lem:sjolin-integral}
Let $\phi$ be given by \eqref{eq:alpha-beta-phi} with $\beta>\alpha$ and let $\rho>0$. Then
$$\int_0^{\delta} \big|\phi^{(n)}(t)\big|^{\rho} dt
\ge C \delta^{-\rho n (\beta+1)+1+\alpha} e^{-\rho \delta^{-\alpha}}
,\ \text{as } \delta\to 0^+.$$
\end{lem}

Applying Lemma \ref{lem:sjolin-integral} with $\rho=\frac{2}{n(n+1)}$ to \eqref{eq:knapp}, we obtain
$$\delta^{-\frac{2}{n+1}(\beta+1)+1+\alpha}\le C \delta^{\frac{n-1}{n+1}}
,\ \text{as } \delta\to 0^+,$$
which gives rise to a contradiction if 
$$\beta>\frac{n+1}{2}\alpha.$$
This completes the proof of Proposition \ref{intro:prop}.

\section{Appendix}\label{appendix}
In this section we give a proof of Lemma \ref{lem:sjolin-integral}. We first observe that, since $\beta>\alpha$, by induction on $n$ we have
$$\phi^{(n)}(t)=e^{-t^{-\alpha}}\left( P(t^{-1})\sin(t^{-\beta})+Q(t^{-1})\cos(t^{-\beta})\right )$$
where $P$ and $Q$ are `polynomials' consisting of fractional powers and 
$$\max\{\text{deg}\, P, \text{deg}\, Q\}=n(\beta+1).$$
Write $u=t^{-\beta}$. Then
$$\phi^{(n)}(t)=e^{-u^{\frac{\alpha}{\beta}}} u^{\frac{n(\beta+1)}{\beta}} \psi(u)$$
where
$$\psi(u)=P_0(u)\sin(u)+Q_0(u)\cos(u)$$
with $P_0$ and $Q_0$ being two `polynomials' as above satisfying
$$\max\{\text{deg}\, P_0, \text{deg}\, Q_0\}=0.$$
For sufficiently large $u$, we can write
$$\psi(u)=\sqrt{P_0(u)^2+Q_0(u)^2} \cos\big(u+\theta(u)\big)$$
where
$$\lim_{u\to\infty} \sqrt{P_0(u)^2+Q_0(u)^2}>0$$
and $\theta(u)$ satisfies
$$\lim_{u\to\infty} \theta'(u)=0,
\quad \lim_{u\to\infty} \theta(u)
=\begin{cases}
0  &\text{ if } \text{deg}\, Q_0=0,\\
-\frac{\pi}{2} &\text{ if } \text{deg}\, P_0=0.
\end{cases}$$
Therefore, changing to the variable $u$, for sufficiently small $\delta$ we have
\begin{align*}
\int_0^{\delta} \big|\phi^{(n)}(t)\big|^{\rho} dt
&=\frac{1}{\beta}\int_{\delta^{-\beta}}^\infty  
e^{-\rho u^{\frac{\alpha}{\beta}}}  u^{\rho\frac{n(\beta+1)}{\beta}-\frac{1}{\beta}-1} |\psi(u)|^\rho du\\
&\ge C \int_{\delta^{-\beta}}^\infty
e^{-\rho u^{\frac{\alpha}{\beta}}}
u^{\rho\frac{n(\beta+1)}{\beta}-\frac{1}{\beta}-1} |\cos\big(u+\theta(u)\big)|^\rho du.
\end{align*}
Changing to the variable $v=u+\theta(u)$, since $\beta>\alpha$, the last integral can be bounded below by
$$C \int_{\delta^{-\beta}+\theta(\delta^{-\beta})}^\infty
e^{-\rho v^{\frac{\alpha}{\beta}}}
v^{\rho\frac{n(\beta+1)}{\beta}-\frac{1}{\beta}-1} |\cos(v)|^\rho dv.$$
Choose $k\in\mathbb N$ such that
$$2(k-1)\pi\le\delta^{-\beta}+\theta(\delta^{-\beta})\le 2k\pi.$$
Then the integral above can be bounded below by
\begin{align*}
\int_{2k\pi}^\infty
e^{-\rho v^{\frac{\alpha}{\beta}}}
v^{\rho\frac{n(\beta+1)}{\beta}-\frac{1}{\beta}-1} |\cos(v)|^\rho dv
\ge & C\sum_{j=k}^\infty
e^{-\rho (2j\pi)^{\frac{\alpha}{\beta}}}
(2j\pi)^{\rho\frac{n(\beta+1)}{\beta}-\frac{1}{\beta}-1}\\
\ge & C\int_{2k\pi}^\infty
e^{-\rho v^{\frac{\alpha}{\beta}}}
v^{\rho\frac{n(\beta+1)}{\beta}-\frac{1}{\beta}-1} dv.
\end{align*}
Letting $w=\rho v^{\frac{\alpha}{\beta}}$, the last integral becomes
$$C\int_{\rho(2k\pi)^{\frac{\alpha}{\beta}}}^\infty
e^{-w}
w^{\rho\frac{n(\beta+1)}{\alpha}-\frac{1}{\alpha}-1} dw.$$
Using integrating by parts, this can be bounded below by
$$C (2k\pi)^{\rho\frac{n(\beta+1)}{\beta}-\frac{1}{\beta}-\frac{\alpha}{\beta}}
e^{-\rho(2k\pi)^{\frac{\alpha}{\beta}}}.$$
Since $2k\pi$ and $\delta^{-\beta}$ differ by at most a constant, we can further bound this below by
$$C \delta^{-\rho{n(\beta+1)}+1+{\alpha}}
e^{-\rho \delta^{-\alpha}},$$
which completes the proof of Lemma \ref{lem:sjolin-integral}.

\bibliographystyle{abbrv}
\bibliography{bibliography}
\nocite{Steinig1971}
\nocite{Fefferman1970,Zygmund1974,Prestini1978,Wright1992,Mockenhaupt,BakLee2004}

\end{document}